\documentclass[12pt]{article}

\usepackage{amsmath,amssymb,amsthm}
\usepackage[margin=2.5cm]{geometry}
\usepackage{hyperref}

\newtheorem{theorem}{Theorem}
\newtheorem{lemma}[theorem]{Lemma}

\theoremstyle{definition}
\newtheorem{definition}[theorem]{Definition}

\theoremstyle{remark}
\newtheorem{remark}[theorem]{Remark}

\newcommand{\R}{\mathbb{R}}
\newcommand{\N}{\mathbb{N}}
\newcommand{\Hk}{\mathbf{H}}
\newcommand{\Pk}{\mathbf{P}}
\newcommand{\ev}{\operatorname{ev}}
\newcommand{\im}{\operatorname{im}}
\newcommand{\rank}{\operatorname{rank}}
\newcommand{\diag}{\operatorname{diag}}
\newcommand{\nrm}[1]{\|#1\|}

\title{Completing the rank identity for Hadamard powers\\of Euclidean distance matrices}
\author{Boris Horvat
\footnote{Abelium d.o.o. and MobIQ d.o.o., Slovenia,
	\texttt{boris.horvat@abelium.com} {\em (Corresponding author)}}\,
\and Alen Orbani\'{c} 
\footnote{Abelium d.o.o., Slovenia,
	\texttt{alen.orbanic@abelium.com} }\,
\and Iztok Kavkler
\footnote{Abelium d.o.o., Slovenia,
	\texttt{iztok.kavkler@abelium.com} }
}
\date{30. 5. 2026}

\begin{document}
\maketitle

\begin{abstract}
Horvat et al.\ (J.\ Math.\ Chem., 2014) showed that the rank of the
$n$-th Hadamard power $D^{(n)}$ of a Euclidean distance matrix satisfies
$\rank D^{(n)} \le R_d^n$, and proved that the inequality is strict whenever
an annihilating polynomial exists.  The converse---that the absence of
annihilating polynomials forces $\rank D^{(n)} = R_d^n$---was left as an
open problem.  We resolve it by exhibiting a kernel factorisation
$D^{(n)} = \Phi_V\, M\, \Phi_V^T$, where $\Phi_V$ is the evaluation matrix on
the polynomial space $V$ and $M$ is a universal matrix independent of
the point configuration.  A trinomial expansion of the kernel reveals
that $M$ has a block-diagonal structure whose blocks are sums of Gram
matrices with positive coefficients; this yields the nonsingularity
of~$M$ and completes the rank identity.
\end{abstract}

\section{Introduction}\label{sec:intro}

The Hadamard (entry-wise) power of a Euclidean distance matrix (EDM) arises naturally
in mathematical chemistry when comparing molecular structures of different sizes via
eigenvalue-based descriptors \cite{HorvatJaklicKavklerRandic2014}.
Given points $\mathbf{x}_1,\dots,\mathbf{x}_m \in \R^d$, the EDM is
$D = [d_{ij}]$ with $d_{ij} = \nrm{\mathbf{x}_i - \mathbf{x}_j}^2$, and its
$n$-th Hadamard power is $D^{(n)} = [d_{ij}^n]$.  Euclidean distance matrices
go back to Schoenberg \cite{Schoenberg1935}; for their rank and nullspace
structure see Alfakih \cite{Alfakih2006}.

Horvat et al.\ \cite{HorvatJaklicKavklerRandic2014} proved:
\begin{itemize}
\item \textbf{Spherical case (Theorem~1):} if the points lie on a sphere, then
  $\rank D^{(n)} \le S_d^n := \binom{d+n-1}{d-1} + \binom{d+n-2}{d-1}$,
  with equality if and only if no polynomial of the appropriate degree
  annihilates all points.
\item \textbf{Generic case (Theorem~2):} for arbitrary points,
  $\rank D^{(n)} \le R_d^n := \binom{d+n}{d} + \binom{d+n-1}{d}$,
  with strict inequality when an annihilating polynomial exists.
\end{itemize}

In Theorem~1 the equality direction follows from Proposition~2 of
\cite{HorvatJaklicKavklerRandic2014}, which requires the summands of an
even/odd-degree decomposition of $D^{(n)}$ to be normal: on a sphere of
radius~$R$ the diagonal matrix $N=\diag(\nrm{\mathbf{x}_i}^2)$ equals
$R^2 I$ after centring, hence is scalar and preserves normality.  For general
configurations the norms vary, $N$ is no longer scalar, the summands are not
normal, and Proposition~2 is unavailable.  Consequently only the inclusion
$\im D^{(n)} \subseteq \ev_X(V)$ was established
in~\cite{HorvatJaklicKavklerRandic2014}, where $V$ is the polynomial space
\begin{equation}\label{eq:V}
  V := \sum_{k=0}^{n} \Hk^k_{\R^d}
     + \sum_{l=0}^{n-1} \nrm{\mathbf{x}}^{2(n-l)}\, \Hk^l_{\R^d}.
\end{equation}
The reverse inclusion is equivalent to the rank identity $\rank D^{(n)}=R_d^n$
when no annihilating polynomial exists, and was left open.

We resolve it by a route that avoids normality altogether: a \emph{kernel
factorisation} $D^{(n)} = \Phi_V\, M\, \Phi_V^T$, where $\Phi_V$ is the
$m \times R_d^n$ evaluation matrix of a basis of $V$ and $M$ is an
$R_d^n \times R_d^n$ matrix determined solely by the kernel
$K(\mathbf{x},\mathbf{y}) = \nrm{\mathbf{x}-\mathbf{y}}^{2n}$ and the choice of
basis, independent of the points.  We then prove that $M$ is nonsingular, from
which the rank identity follows immediately.

\section{Preliminaries}\label{sec:prelim}

We follow the notation of \cite{HorvatJaklicKavklerRandic2014}.  Throughout,
$d$ and $n$ denote positive integers.
Let $\Pk_{\R^d}$ denote the space of all polynomials on $\R^d$ with real
coefficients, $\Pk^n_{\R^d}$ the subspace of polynomials of degree at most~$n$,
and $\Hk^n_{\R^d}$ the subspace of \emph{homogeneous} polynomials of degree~$n$
(throughout, $\Hk^n$ means homogeneous, not harmonic).  Then
$\dim \Hk^n_{\R^d} = \binom{n+d-1}{d-1}$ and
$\dim \Pk^n_{\R^d} = \binom{n+d}{d}$.

Throughout, $\N = \{0, 1, 2, \dots\}$ includes~$0$.  A \emph{multi-index} is a
vector $\alpha = (\alpha_1, \dots, \alpha_d)
\in \N^d$; we write $|\alpha| = \alpha_1 + \cdots + \alpha_d$ and
$\mathbf{x}^\alpha = x_1^{\alpha_1}\cdots x_d^{\alpha_d}$.  For a multi-index
lower argument we use the multinomial coefficient
$\binom{m}{\gamma} = m!/(\gamma_1!\cdots\gamma_d!)$, read as $0$ unless every
$\gamma_i \in \N$ and $\sum_i \gamma_i = m$.

For a matrix $X = [\mathbf{x}_1 \cdots \mathbf{x}_m] \in \R^{d \times m}$
and a polynomial $p \in \Pk_{\R^d}$, the \emph{evaluation vector} is
$\ev_X(p) = [p(\mathbf{x}_1), \dots, p(\mathbf{x}_m)]^T \in \R^m$;
for a subspace $P \subseteq \Pk_{\R^d}$ we write
$\ev_X(P) = \{\ev_X(p) : p \in P\}$.

\section{Main result}\label{sec:main}

We partition the polynomial space $V$ into three degree-compatible blocks
(Section~\ref{sec:basis}), expand $K$ via a trinomial identity to obtain
$D^{(n)}=\Phi_V M\Phi_V^T$ (Section~\ref{sec:factorisation}), and exploit the
resulting block-diagonal structure of $M$ to prove it is nonsingular
(Sections~\ref{sec:block-structure}--\ref{sec:nonsingular}).

\subsection{The polynomial space and its basis}\label{sec:basis}

Each basis function of~$V$ has the form
$\varphi_{r,\alpha}(\mathbf{x}) = \nrm{\mathbf{x}}^{2r}\, \mathbf{x}^\alpha$, of
homogeneous degree $2r + |\alpha|$.  We partition~$V$ by degree into three
blocks, anticipating that the jointly homogeneous kernel pairs only degrees
summing to~$2n$ (Section~\ref{sec:block-structure}).

\begin{definition}[Three-block partition of $V$]\label{def:blocks}
\
\begin{itemize}
\item[\textbf{Block~A}] ($\Pk^{<n}_{\R^d}$, degrees $0, \dots, n{-}1$):
$\varphi_{0,\alpha}(\mathbf{x}) = \mathbf{x}^\alpha$
for $|\alpha| \le n-1$, with $|A| = \binom{n+d-1}{d}$.

\item[\textbf{Block~B}] (homogeneous degree $n$):
$\varphi_{0,\alpha}(\mathbf{x}) = \mathbf{x}^\alpha$
for $|\alpha| = n$, with $|B| = \binom{n+d-1}{d-1}$.

\item[\textbf{Block~C}] (mixed terms, degrees $n{+}1, \dots, 2n$):
$\varphi_{r,\alpha}(\mathbf{x}) = \nrm{\mathbf{x}}^{2r}\, \mathbf{x}^\alpha$
for $r = 1, \dots, n$ and $|\alpha| = n - r$, with
$|C| = \sum_{r=1}^{n} \binom{n-r+d-1}{d-1} = \binom{n+d-1}{d}$.
\end{itemize}
\end{definition}

Note that $|A| = |C|$.  Blocks~A and~B together are the monomial basis of
$\Pk^n_{\R^d}$, while Block~C spans the second summand of $V$ (substituting
$l=n-r$).  These functions are linearly independent, hence a basis of~$V$.
Indeed, each $\varphi_{r,\alpha}$ is homogeneous, of degree $0,\dots,n-1$ in
Block~A, $n$ in Block~B, and $n+r$ for the part of Block~C with parameter
$r=1,\dots,n$; these degrees are pairwise distinct, so a linear dependence would
have to occur within a single degree.  Within Blocks~A and~B the
$\mathbf{x}^\alpha$ are independent, and the degree-$(n+r)$ part of Block~C is
the image of the monomial basis $\{\mathbf{x}^\alpha:|\alpha|=n-r\}$ of
$\Hk^{n-r}_{\R^d}$ under multiplication by the nonzero polynomial
$\nrm{\mathbf{x}}^{2r}$, which is injective since $\R[x_1,\dots,x_d]$ is an
integral domain; hence that part is independent and of full dimension
$\binom{n-r+d-1}{d-1}$.  The total dimension is
\[
  N := |A| + |B| + |C| = \binom{n+d}{d} + \binom{n+d-1}{d} = R_d^n.
\]
We order the basis as $\varphi_1, \dots, \varphi_N$, grouped by block (A, then
B, then C): within Block~A by ascending homogeneous degree, within Block~C by
\emph{descending} homogeneous degree, and lexicographically among multi-indices
of equal degree (and lexicographically within Block~B).  With this ordering the
degree-$k$ part of Block~A ($k=0,\dots,n-1$) occupies the same block position as
the degree-$(2n-k)$ part of Block~C, so the coupling block $M_{AC}$ of
Section~\ref{sec:block-structure} is literally block-diagonal.  We
define the \emph{evaluation matrix}
$\Phi_V = [\varphi_j(\mathbf{x}_i)]_{i,j} \in \R^{m \times N}$.

\subsection{The kernel factorisation}\label{sec:factorisation}

We express $K(\mathbf{x},\mathbf{y})=\nrm{\mathbf{x}-\mathbf{y}}^{2n}$ as a
bilinear form in the basis of~$V$, thereby obtaining $D^{(n)}=\Phi_V M\Phi_V^T$.
Writing $\nrm{\mathbf{x}-\mathbf{y}}^2 = \nrm{\mathbf{x}}^2 -
2\,\mathbf{x}^T\mathbf{y} + \nrm{\mathbf{y}}^2$, the multinomial theorem gives
the \emph{trinomial expansion}
\begin{equation}\label{eq:trinomial}
  K(\mathbf{x},\mathbf{y})
  = \sum_{a+b+c=n}\frac{n!}{a!\,b!\,c!}\;
    \nrm{\mathbf{x}}^{2a}\,(-2\,\mathbf{x}^T\mathbf{y})^b\,\nrm{\mathbf{y}}^{2c}.
\end{equation}

\begin{lemma}\label{lem:factorisation}
There exists a matrix $M \in \R^{N \times N}$, depending only on $d$ and $n$,
such that $D^{(n)} = \Phi_V\, M\, \Phi_V^T$.
\end{lemma}

\begin{proof}
Expanding $(-2\,\mathbf{x}^T\mathbf{y})^b
= (-2)^b\sum_{|\rho|=b}\binom{b}{\rho}\,\mathbf{x}^\rho\,\mathbf{y}^\rho$
in~\eqref{eq:trinomial}, each summand contributes the $\mathbf{x}$-factor
$\nrm{\mathbf{x}}^{2a}\mathbf{x}^\rho$ and the $\mathbf{y}$-factor
$\nrm{\mathbf{y}}^{2c}\mathbf{y}^\rho$.  We show every $\mathbf{x}$-factor lies
in~$V$; by symmetry so does every $\mathbf{y}$-factor, whence $K \in V \otimes V$
and the factorisation follows.  Fix $(a,b,c)$ with $a+b+c=n$ and $|\rho|=b$; the
$\mathbf{x}$-factor has homogeneous degree $2a+b$.

\begin{description}
\item[Case~1 ($2a+b\le n$):]
$\nrm{\mathbf{x}}^{2a}\mathbf{x}^\rho$ is a polynomial of degree $\le n$, hence
lies in $\Pk^n_{\R^d}$, spanned by Blocks~A and~B.

\item[Case~2 ($2a+b>n$):]
since $b=n-a-c$ we have $2a+b=n+(a-c)$; set $j=a-c\ge 1$.  Then
\[
  \nrm{\mathbf{x}}^{2a}\mathbf{x}^\rho
  = \nrm{\mathbf{x}}^{2j}\,
    \underbrace{\nrm{\mathbf{x}}^{2c}\mathbf{x}^\rho}_
    {\text{homogeneous of degree }n-j},
\]
so each monomial $\mathbf{x}^\sigma$ ($|\sigma|=n-j$) of the inner factor,
multiplied by $\nrm{\mathbf{x}}^{2j}$, gives the Block~C basis element
$\varphi_{j,\sigma}$.
\end{description}

\noindent
Collecting coefficients over all $(a,b,c,\rho)$ gives
$K(\mathbf{x},\mathbf{y})=\sum_{i,j}M_{ij}\,\varphi_i(\mathbf{x})\,
\varphi_j(\mathbf{y})$; evaluating at point pairs yields
$D^{(n)}=\Phi_V\,M\,\Phi_V^T$.
\end{proof}

\subsection{Block structure of \texorpdfstring{$M$}{M}}\label{sec:block-structure}

The kernel is jointly homogeneous of degree~$2n$
($K(t\mathbf{x},t\mathbf{y})=t^{2n}K(\mathbf{x},\mathbf{y})$), so
\begin{equation}\label{eq:M-homogeneity}
  M_{ij} \ne 0 \implies \deg(\varphi_i) + \deg(\varphi_j) = 2n.
\end{equation}
Within the A/B/C partition the only degree pairings summing to~$2n$ are
A--C (degree $k<n$ with $2n-k$) and B--B (degree $n$ with $n$); all other
blocks vanish, giving
\begin{equation}\label{eq:M-block}
  M =
  \begin{pmatrix}
    0  & 0      & M_{AC} \\[3pt]
    0  & M_{BB} & 0 \\[3pt]
    M_{CA} & 0  & 0
  \end{pmatrix},
  \qquad M_{CA}=M_{AC}^T,\quad |A|=|C|.
\end{equation}
Since~\eqref{eq:M-homogeneity} pairs each degree-$k$ Block~A index with
degree-$(2n-k)$ Block~C indices, $M_{AC}=\bigoplus_{k=0}^{n-1} M_{AC}^{(k)}$ is
block-diagonal, with $M_{AC}^{(k)}$ square of order $\binom{k+d-1}{d-1}$.
Reordering the row blocks $(A,B,C)\to(C,B,A)$ leaves $B$ fixed and interchanges
the two equal-size end blocks, a permutation of parity $(-1)^{|A|}$
(independent of~$|B|$), so
\begin{equation}\label{eq:det-M}
  \det M = (-1)^{|A|}\,\det(M_{BB})\,(\det M_{AC})^2,
\end{equation}
and $M$ is nonsingular if and only if each $M_{AC}^{(k)}$ ($k<n$) and $M_{BB}$ is.

\subsection{Nonsingularity of \texorpdfstring{$M$}{M}}\label{sec:nonsingular}

Set
\[
  M^{(k)} := M_{AC}^{(k)}\ (k=0,\dots,n-1), \qquad M^{(n)} := M_{BB},
\]
a square matrix of order $\binom{k+d-1}{d-1}$ indexed on both sides by
$\{\alpha\in\N^d : |\alpha|=k\}$.  Two lemmas settle nonsingularity at once.

\begin{lemma}[Weighted Gram matrices]\label{lem:gram}
Let $I$, $J$ be finite sets, let $w_\rho > 0$ for $\rho \in I$, and let
$u_\rho(\alpha) \in \R$ for $\rho \in I$, $\alpha \in J$.  Then
$P_{\alpha,\beta} = \sum_{\rho \in I} w_\rho\, u_\rho(\alpha)\, u_\rho(\beta)$
defines a symmetric positive semidefinite $|J| \times |J|$ matrix.
\end{lemma}

\begin{proof}
With $U_{\rho,\alpha} = u_\rho(\alpha)$ and $W = \diag(w_\rho)$ we have
$P = U^T W U$, which is symmetric ($W$ diagonal) and satisfies
$v^T P v = \|W^{1/2}Uv\|^2 \ge 0$.
\end{proof}

\begin{lemma}[Uniform positivity of degree blocks]\label{lem:block-positive}
For each $k = 0, 1, \dots, n$, the matrix $M^{(k)}$ is symmetric and
$(-1)^k M^{(k)} \succ 0$.  In particular, $M$ is nonsingular.
\end{lemma}

\begin{proof}
The $(a,b,c)$-summand of~\eqref{eq:trinomial} has $\mathbf{x}$-degree $2a+b$ and
$\mathbf{y}$-degree $b+2c$, so its degree-$(k,\,2n-k)$ part needs $2a+b=k$ and
$b+2c=2n-k$, giving $a=(k-b)/2$, $c=(2n-k-b)/2$; integrality and nonnegativity
force $b\equiv k\pmod 2$ and $0\le b\le k$.  Thus
\begin{equation}\label{eq:K-k-restriction}
  K\big|_{(k,\,2n-k)}
  = \nrm{\mathbf{y}}^{2(n-k)}
    \sum_{\substack{0\le b\le k\\b\equiv k\,(2)}}
    d_b\;(\mathbf{x}^T\mathbf{x})^{(k-b)/2}(\mathbf{x}^T\mathbf{y})^b(\mathbf{y}^T\mathbf{y})^{(k-b)/2},
  \qquad
  d_b = \frac{n!\,(-2)^b}{\bigl(\tfrac{k-b}{2}\bigr)!\;b!\;
        \bigl(\tfrac{2n-k-b}{2}\bigr)!},
\end{equation}
using $2c=2n-k-b=2(n-k)+(k-b)$ with $k-b$ even.  The prefactor
$\nrm{\mathbf{y}}^{2(n-k)}$ is the Block~C weight:
the degree-$(2n-k)$ basis elements of Block~C are exactly
$\varphi_{n-k,\beta}=\nrm{\mathbf{y}}^{2(n-k)}\mathbf{y}^\beta$ with $|\beta|=k$
(the weight being $1$, i.e.\ the Block~B monomials $\mathbf{y}^\beta$, when
$k=n$), so after factoring out $\nrm{\mathbf{y}}^{2(n-k)}$ the entry
$M^{(k)}_{\alpha\beta}$ is the coefficient of $\mathbf{x}^\alpha\mathbf{y}^\beta$
in $\sum_b d_b\,(\mathbf{x}^T\mathbf{x})^{(k-b)/2}(\mathbf{x}^T\mathbf{y})^b(\mathbf{y}^T\mathbf{y})^{(k-b)/2}$.
Each summand
$(\mathbf{x}^T\mathbf{x})^{(k-b)/2}(\mathbf{x}^T\mathbf{y})^b(\mathbf{y}^T\mathbf{y})^{(k-b)/2}$ is
invariant under $\mathbf{x}\leftrightarrow\mathbf{y}$, so this coefficient equals
that of $\mathbf{x}^\beta\mathbf{y}^\alpha$; hence $M^{(k)}$ is symmetric.

Expanding in monomials, the coefficient of $\mathbf{x}^\alpha\mathbf{y}^\beta$
(with $|\alpha|=|\beta|=k$) in the $b$-th summand is
\begin{equation}\label{eq:Pb}
  [P_b]_{\alpha,\beta}
  = \sum_{|\rho|=b}\binom{b}{\rho}\,
    \binom{(k{-}b)/2}{(\alpha{-}\rho)/2}\,
    \binom{(k{-}b)/2}{(\beta{-}\rho)/2},
\end{equation}
which is PSD by Lemma~\ref{lem:gram} (with $w_\rho=\binom{b}{\rho}$ and
$u_\rho(\alpha)=\binom{(k-b)/2}{(\alpha-\rho)/2}$, using the multi-index
convention of Section~\ref{sec:prelim}).  Since $b\equiv k\pmod 2$ we have
$(-1)^k\,(-2)^b = 2^b$, so $(-1)^k d_b>0$ for every valid~$b$, and
$(-1)^k M^{(k)}=\sum_{\substack{0\le b\le k\\b\equiv k\,(2)}} (-1)^k d_b\,P_b$ is
a positively weighted sum of PSD matrices.  The $b=k$ term dominates: there $(k-b)/2=0$, so
$u_\rho(\alpha)=\delta_{\alpha,\rho}$ and
$[P_k]_{\alpha,\beta}=\binom{k}{\alpha}\delta_{\alpha\beta}$, with coefficient
$(-1)^k d_k=2^k\binom{n}{k}$.  Therefore
\[
  (-1)^k M^{(k)} \;\succcurlyeq\; 2^k\binom{n}{k}\,\diag\!\bigl(\tbinom{k}{\alpha}\bigr)
  \;\succ\; 0,
\]
and nonsingularity of $M$ follows from~\eqref{eq:det-M}.
\end{proof}

\begin{remark}
For $d=1$ the basis is $\{1, t, \dots, t^{2n}\}$ and $M$ is the anti-diagonal
matrix with $M_{a,\,2n-a}=(-1)^a\binom{2n}{a}$, so
$\det M=\prod_{a=0}^{2n}\binom{2n}{a}>0$, the explicit $d=1$ instance of
$(-1)^k M^{(k)}\succ0$.
\end{remark}

\subsection*{The rank identity}

\begin{theorem}\label{thm:main}
Let $\mathbf{x}_1, \dots, \mathbf{x}_m \in \R^d$ be distinct points with
Euclidean distance matrix $D = [\nrm{\mathbf{x}_i-\mathbf{x}_j}^2]$ and Hadamard
power $D^{(n)} = [\nrm{\mathbf{x}_i-\mathbf{x}_j}^{2n}]$, and let
$R_d^n := \binom{d+n}{d} + \binom{d+n-1}{d}$.  Then $\rank D^{(n)} \le R_d^n$,
with strict inequality if and only if some nonzero element of the space~$V$
of~\eqref{eq:V} vanishes at every~$\mathbf{x}_i$.  Equivalently,
the inequality is strict if and only if there exist
$p \in \Pk^n_{\R^d}$ and $q \in \Pk^{n-1}_{\R^d}$, not both zero, with
\[
  p(\mathbf{x}_i) + \nrm{\mathbf{x}_i}^{2n}\,
  q\!\left(\mathbf{x}_i/\nrm{\mathbf{x}_i}^2\right) = 0
  \qquad (i = 1, \dots, m),
\]
where $\nrm{\mathbf{x}}^{2n} q(\mathbf{x}/\nrm{\mathbf{x}}^2)$ denotes the
polynomial in $\sum_{l=0}^{n-1}\nrm{\mathbf{x}}^{2(n-l)}\Hk^l_{\R^d}$ to which it
extends: writing $q=\sum_{l=0}^{n-1}q_l$ with $q_l\in\Hk^l_{\R^d}$, the $l$-th
term is $\nrm{\mathbf{x}}^{2n} q_l(\mathbf{x}/\nrm{\mathbf{x}}^2)
=\nrm{\mathbf{x}}^{2(n-l)}q_l(\mathbf{x})$, carrying the positive power $n-l\ge1$
of $\nrm{\mathbf{x}}^2$ (so the extension is well defined at $\mathbf{x}=0$).
\end{theorem}

\begin{proof}
By Lemma~\ref{lem:factorisation}, $\im D^{(n)}\subseteq\im\Phi_V$, so
$\rank D^{(n)}\le\rank\Phi_V\le N=R_d^n$ (the last step since $\Phi_V$ has $N$
columns).  If an annihilating $f\in V$ exists, then $\ker(\Phi_V)\ne\{0\}$,
hence $\rank\Phi_V<N$ and the inequality is strict.  Conversely, suppose no
annihilating element exists: then $\ker(\Phi_V)=\{0\}$, so $\Phi_V$ is injective
with full column rank $N=R_d^n$ (in particular $m\ge N$).

By Lemma~\ref{lem:factorisation}, $D^{(n)} = \Phi_V M \Phi_V^T$, and by
Lemma~\ref{lem:block-positive}, $M$ is nonsingular.  We claim
$\ker(D^{(n)}) = \ker(\Phi_V^T)$: the inclusion $\supseteq$ is immediate, and
if $\Phi_V M \Phi_V^T v=0$ then injectivity of $\Phi_V$ gives $M\Phi_V^T v=0$,
whence $\Phi_V^T v=0$ since $M$ is nonsingular.  As $\Phi_V$ has full column
rank, $\dim\ker(\Phi_V^T)=m-N$, so
\[
  \rank D^{(n)} = m - \dim\ker(D^{(n)}) = m-(m-N) = R_d^n. \qedhere
\]
\end{proof}

\begin{remark}\label{rem:geometric}
For generic configurations the condition $\ker(\Phi_V)=\{0\}$ holds as soon as
$m \ge R_d^n$.  Indeed, for any basis $f_1,\dots,f_N$ of the $N$-dimensional
space $V$ there exist points $\mathbf{y}_1,\dots,\mathbf{y}_N$ with
$\det[f_j(\mathbf{y}_i)]_{i,j=1}^N\ne0$, chosen inductively: if the evaluation
rows obtained so far spanned only a proper subspace $W\subsetneq\R^N$ and every
further row $[f_1(\mathbf{y}),\dots,f_N(\mathbf{y})]$ lay in~$W$, then a nonzero
$c\in W^\perp$ would give $\sum_j c_j f_j(\mathbf{y})=0$ for all
$\mathbf{y}\in\R^d$; as $\R$ is infinite this forces $\sum_j c_j f_j$ to be the
zero polynomial, contradicting the independence of the~$f_j$.  Hence some point
raises the rank, and after $N$ steps the evaluation matrix is nonsingular.
Consequently at least one $N\times N$ minor of $\Phi_V$ is a not-identically-zero
polynomial in the coordinates of the points, and the rank-deficient
configurations lie in the common zero set of these finitely many minors---a
proper algebraic subset of $(\R^d)^m$, hence of Lebesgue measure zero (and of
probability zero under any product of absolutely continuous distributions
on~$\R^d$).
\end{remark}

\section{Conclusion}\label{sec:conclusion}

We have completed Theorem~2 of \cite{HorvatJaklicKavklerRandic2014}: when no
annihilating polynomial exists, $\rank D^{(n)} = R_d^n$ exactly.  The proof rests
on the universal kernel factorisation $D^{(n)} = \Phi_V M \Phi_V^T$ (with $M$
independent of the point configuration) together with the uniform block
positivity $(-1)^k M^{(k)}\succ 0$, which yields nonsingularity of~$M$ and
reduces the rank question to a finite-dimensional, configuration-independent
positivity problem---thereby bypassing the normality obstruction
(Section~\ref{sec:intro}) that blocks the even/odd decomposition
of~\cite{HorvatJaklicKavklerRandic2014} for general points.  A NumPy
implementation confirms the block structure~\eqref{eq:M-block},
$\rank M = R_d^n$, and the predicted drop to $S_d^n$ on spherical
configurations, for all $d,n \le 7$ where dense computation is feasible.

A natural open problem is a closed-form expression for $\det M$ for general $d$
and $n$, analogous to the product formula
$\det M=\prod_{a=0}^{2n}\binom{2n}{a}$ of the $d=1$ case; the numerical data suggest a
structured factorisation in terms of multinomial coefficients.

\end{document}